\documentclass[preprint]{imsart}

\RequirePackage[OT1]{fontenc}
\RequirePackage{amsthm,amsmath}
\usepackage{amssymb}
\RequirePackage[numbers]{natbib}
\RequirePackage[colorlinks,citecolor=blue,urlcolor=blue]{hyperref}

\usepackage{amsmath}
\usepackage{datetime}
\usepackage{enumerate}
\usepackage{xcolor}
\usepackage{graphicx}
\usepackage{epstopdf}
\numberwithin{equation}{section}
\newtheorem{theorem}{Theorem}[section]
\newtheorem{lemma}{Lemma}[section]
\newtheorem{corollary}{Corollary}[section]

\newtheorem{remark}{Remark}[section]

\newcommand{\COM}[1]{}
\newcommand{\dsum}{\displaystyle\sum}
\newcommand{\dint}{\displaystyle\int}

\newcommand{\cid}{\stackrel{d}{\longrightarrow}}

\newcommand{\civ}{\stackrel{v}{\longrightarrow}}
\newcommand{\cij}{\stackrel{J_1}{\longrightarrow}}
\newcommand{\toi}{\to\infty}
\newcommand{\eind}{\stackrel{d}{=}}

    \newcommand{\R}{\mathbb{R}}

    \newcommand{\1}{\mathbb{I}}

    \newcommand{\Sb}{\mathbb{S}}
    \newcommand{\Sc}{\mathcal{S}}
    \newcommand{\Leb}{\mathrm{Leb}}
    \newcommand{\Exp}{\operatorname{E}}

    \newcommand{\vep}{\varepsilon}

\newcommand{\bfV}{{\boldsymbol{V}}}

\begin{document}

\begin{frontmatter}
\title{Limits of renewal processes and Pitman-Yor distribution}
\runtitle{Limits of renewal processes and Pitman-Yor distribution}

\begin{aug}
\author{\fnms{Bojan} \snm{Basrak}\ead[label=e1]{bbasrak@math.hr}}

%IF FOOTNOTE t1 NEEDED \author{\fnms{Bojan} \snm{Basrak}\thanksref{t1}\ead[label=e1]{bbasrak@math.hr}}
%\and
%\author{\fnms{Drago} \snm{\v Spoljari\'c}\ead[label=e2]{drago.spoljaric@rgn.hr}}

%\thankstext{t1}{Corresponding author}
%\thankstext{t2}{First supporter of the project}
%\thankstext{t3}{Second supporter of the project}
\runauthor{B.Basrak}

\affiliation{University of Zagreb} %\thanksmark{m1}}

\address{
Department of Mathematics\\
University of Zagreb\\
Bijeni\v cka 30, Zagreb\\
Croatia\\
\printead{e1}\\
\phantom{E-mail:\ }}

%\address{D. \v Spoljari\'c\\
%Faculty of Mining, Geology and Petroleum Engineering\\
%University of Zagreb\\
%Pierottijeva 6, Zagreb\\
%Croatia\\
%\printead{e2}\\
%}
\end{aug}

\begin{abstract}
We consider a renewal process with regularly varying 
stationary and weakly dependent steps, and
prove that the steps  made before a given time $t$, satisfy an interesting invariance principle.
Namely, together with the age of the renewal process at time $t$,   they 
converge after  scaling  to the  Pitman--Yor distribution. 
We further discuss how our results extend the classical Dynkin--Lamperti theorem.

\end{abstract}



\end{frontmatter}

\section{Introduction}

%\COM{TODO: Billingsley, Whitt reference, ,DONE:  O'Brien, LLR or LR, Seneta/Resnick,  Krizmanic strong mix}
%%\section{Model}
By one of the main results in renewal theory, it is known that the age of a renewal process 
has a limiting distribution, % described by the renewal theorem,
given that its  steps have finite mean.
%
%In one of its forms, 
%  the renewal theorem gives the limiting distribution of  the age of a renewal process, provided that its steps have finite mean.
When the 
steps are iid and regularly varying with infinite mean, the limiting distribution is determined  by the 
Dynkin--Lamperti theorem.
 In this article, we aim to understand the limiting behavior of the whole path of such a renewal process before a given time $t$. Moreover, we do that  under milder conditions, that is, we keep the regular variation assumption, but allow certain degree of dependence
between the steps of the renewal process.
More precisely, we  assume that the steps form a stationary sequence $(Y_n)$
of nonnegative random variables which are regularly
varying with index $\alpha \in (0,1)$.
%
% consider a renewal process with stationary, possibly dependent, nonnegative  steps $(Y_t)$ which are  regularly varying
%with index $\alpha \in (0,1)$. 
By one
characterization  of regular variation, see Resnick~\cite{Re87},  this means that there exists a sequence of  nonnegative real numbers $(d_n)$ such that
\begin{equation} \label{e:RegVar}
  n P (Y \in d_n  \;\cdot\;) \civ  \mu (\;\cdot\;) \,, 
 \end{equation}
 as $n\toi$,
where $\civ$ denotes vague convergence of measures on $(0,\infty)$ and the limiting measure
satisfies $\mu(x,\infty) = x^{-\alpha}$ 
for all $x >0$.

It will be useful in the sequel to extend $(d_n)$ to  a function on $[0,\infty)$  by
 denoting $d(t)=d_{\lfloor t\rfloor}$, for $t \geq 0$, with $d_0=1$.  
It is  known that $d$ has an asymptotic inverse, $\widetilde{d}$ say, 
 see Seneta~\cite{Seneta76},
  in the sense that
\begin{equation}\label{d_asympt_eq}
d(\widetilde{d}(t))\sim \widetilde{d}(d(t))\sim t\,,
\end{equation}
as $t\toi$. One can show that
%\begin{equation}\label{def_d_tilde}
%\widetilde{d}(t) \sim \left. \frac{1}{1-F_Y(t)} \right. \,.
%\end{equation} 
%In particular, 
 $\widetilde{d}$   is a regularly varying function with index $\alpha$. 

%\COM{CHATTING, IZBACIO FCLT}
%In order to characterize asymptotic behavior of the renewal steps $(Y_t)$ over a large 
%time interval $[0,t)$, one has to restrict the dependence between the steps in some way.
%Naturally, the iid case is  the easiest to study.
%It is well known  that iid steps $(Y_n)$ which are regularly varying in the sense
% of \eqref{e:RegVar}  with $\alpha \in(0,1)$ satisfy  the functional limit theorem
%%  holds
%%\begin{equation}\label{e:ConvFLT}
%%S_t(\;\cdot\;) = \sum_{i=1}^{\lfloor \widetilde{d} (t) \cdot\rfloor} \frac{Y_i}{t}
%%\cid S (\;\cdot\;)\,, % = \sum_{T_i\leq  \cdot } Z_i \,,
%%\end{equation}
%%where $S$ is a 
%where the limit is a
%stable subordinator or nonnegative stable L\'evy process with  L\'evy  measure
%$\mu$, see Section~\ref{SecPPPY}.
%, and the convergence takes place in Skorohod's $J_1$ topology on the space of
%c\`adl\`ag functions $D[0,\infty)$. 

 Denote by
\begin{equation}\label{e:renewal_process}
\tau(t) = \inf \{ k: Y_1+\cdots + Y_k >t \}\,,
\quad \mbox{ for } t\geq 0\,,
\end{equation} 
 the first passage time of the level $t$
 by the random walk with steps $(Y_n)$.
%\COM{ some text was here, probably not needed}
%In particular, 
%one could guess that in iid case 
%the behavior of the renewal process during
%time interval $[0,t]$ can be captured by the behavior of the 
%partial sum process $(S_t(\cdot))$ during a random time interval $[0,S_t^{\leftarrow}(1)]$
Our main goal is to describe the asymptotics of all  the  steps
in the renewal process before the passage time $\tau(t)$, i.e. of the
random variables
\begin{equation}\label{PrevSteps}
   \frac{Y_i}{t}\,, \quad  i= 1,\ldots , \tau(t) -1 \,, %\left( \right)\,,
\end{equation}
including the age of the renewal process at the passage time, that is
\begin{equation}\label{Meandar}
A^{(t)}  = \frac{t - \sum_{i<\tau(t)} Y_i}{t}\,. %^{\tau(t)-1}
\end{equation}
%This value is also called the length of the last meander interval, cf. Pitman~\cite{Pi06}. 
\COM{Observe that we allow some degree of dependence between the steps 
$(Y_n)$.} 
For iid steps $(Y_n)$,  the proof of the following classical theorem can be found in Bingham et al.~\cite{BGT}.

\begin{theorem}\label{thm_DL} (Dynkin--Lamperti)
Suppose that $(Y_n)$ is  iid sequence of random variables, then
$(Y_n)$ satisfies~\eqref{e:RegVar} with the tail index $\alpha \in (0,1)$
if and only if 
$$
  A^{(t)} \cid  A\,,
$$
as $t\toi$,
where the random variable on the right hand side has a generalized arcsine distribution with the density
\begin{equation}\label{ArcSDens}
q_\alpha (u) = \frac{\sin \pi \alpha}{\pi}  u ^{-\alpha} (1-u)^{\alpha -1 }\,,\quad
u \in [0,1]\,.
\end{equation}
\end{theorem}

% \eqref{e:RegVar} holds if and only if the random
%variables in \eqref{Meandar} converge
%towards a generalized arcsine law  as $t \toi$. 
 It turns out that the necessity part of this theorem holds  for certain dependent renewal processes too.
 More importantly, in all such cases one can describe the joint asymptotic behavior of  the
 random variables in \eqref{PrevSteps}  and \eqref{Meandar}, and show that they, when ordered, form a sequence which converges towards the so-called
Pitman--Yor distribution.
As far as we know, this result is new
even in the iid case.

 The paper is organized as follows: in Section \ref{SecPPPY} we 
 consider Pitman--Yor distribution on the interval partitions
 from the perspective of point processes theory.
  We further present two limiting theorems about stationary
 strongly mixing sequences $(Y_n)$ satisfying \eqref{e:RegVar} which are likely 
 to be of independent interest.
 These theorems are used in Section~\ref{Sec:main} to determine the asymptotic distribution of
 the 
  steps  ${Y_i}/{t}\,,\  i= 1,\ldots , \tau(t) -1$
  and the age of the renewal process $A^{(t)}$. We also exhibit how this result extends the classical 
  Dynkin--Lamperti theorem and discuss corresponding assumptions. It immediately yields the joint asymptotic distribution for the 
  ranked lengths of excursions in a simple symmetric random walk, cf.
   Cs\'aki and Hu~\cite{CsHu}. 
More technical proofs and  results concerning Skorohod's topology and convergence of point measures are postponed to the Appendix.

\section{Point processes and Pitman-Yor distribution} \label{SecPPPY}

In a remarkable series of papers: \cite{PiYo92},  \cite{PPY}, \cite{Pe93}; Perman, Pitman and  Yor  describe the distribution of jumps of stable subordinators on a given time interval. In  particular,  Pitman and Yor in~\cite{PiYo92}, use such jumps to introduce a new family  of distributions on interval partitions and relate them to the classical arcsine laws for Brownian motion. Recall that a stable subordinator $(S(t))_{t\geq 0}$ is a L\'evy process with the Laplace transform given by
 the formula
 $$
  E e^{-\lambda S(t) }= \exp\left[ -t \int_0^\infty \left( 1- e^{-\lambda x} \right)\mu'(d x) \right]\,,
 $$  \COM{Bertoin,Perman}
with the L\'evy measure 
$$
  \mu'(dx) = c_\alpha x^{-\alpha -1} dx\,,
 $$
 for some $\alpha \in (0,1)$ and a constant $c_\alpha>0$ which turns out to be unimportant in the 
 sequel. So without loss of generality we typically assume $c_\alpha =1$, i.e. $\mu'=\mu$.
The subordinator  $(S(t))$ has the distribution of the inverse local 
times of $d$--dimensional Bessel process, for $d= 2(1-\alpha)$, with the case $\alpha=1/2$ corresponding to  the Brownian motion. In other words, jumps of the process $(S(t))$ correspond to the lengths of excursions of the Brownian motion or, more generally Bessel process, away form the origin.
By Ito's representation $(S(t))$ can be constructed from a Poisson process $N$ on the space $[0,\infty)\times (0,\infty]$ with 
intensity measure equal to $\Leb \times \mu'$, so that
\begin{equation}\label{ItoRep}
N= \dsum_i \delta_{T_i,P_i}\ \mbox{ and }\  S (t)  = \dsum_{T_i\leq t} P_i,\quad
\mbox{ for } t\in [0,\infty)\,.
\end{equation} \COM{provjeri $S(t) or S_t$ Bertoin pro1.3 wrong?}
 We alternatively say that $N$ is a Poisson random measure and  denote this by $N\sim\mbox{PRM}(\Leb\times \mu')$.
By the construction, $(S(t))$ is a nondecreasing element of  the space of c\`adl\`ag functions $D[0,\infty)$.
Denote by $z^{\leftarrow}$  the  right continuous generalized inverse  of a function $z\in D[0,\infty)$, i.e.
\begin{equation*}
z^{\leftarrow}(u)=\inf\{s\in[0,\infty):z(s)>u\} \,, \quad u \geq 0\,.
\end{equation*}
 The generalized inverse of the process $S(t)$ is the process
$$
 L(s) = S^{\leftarrow}(s) = \inf \{ t: S(t) > s \}\,,\quad s\geq 0\,.
$$
It is well defined and continuous at any $s\geq 0$, and for the reasons explained above it is called local time process by Bertoin in~\cite{BeSub}. 
If we denote by $\mathcal{Z}$ the closure of the range of the process $(S(t))$, the maximal open subintervals in the set $\mathcal{Z}^c \cap (0,s),\ s>0$, correspond to the jumps
of the subordinator before it crosses over level $s$. Their lengths are
$
 (S(t)- S(t-))\,,\  t < L(s),
$ % or using the representation in \eqref{ItoRep}, they are
which are equal to
$$
 P_i\,,\quad T_i < L(s),
$$ 
above, together with the last incomplete jump which has the length 
\begin{equation}\label{LastInt}
 A_s= s-\dsum_{T_i <  L(s)} P_i\,,
\end{equation}
see Bertoin~\cite{BeSub}.  Considering these points in descending order we arrive at the sequence 
$$
 \bfV(s) = \left( V_1(s),V_2(s),V_3(s),\ldots \right) \,.
$$
Observe that the distribution of $\bfV(s)$ corresponds to the distribution of the point process
\COM{cf Kallenberg Ch III, th 3.3} 
$$
%N_{\bfV(s)} = 
 \dsum_{i} \delta_{V_i{(s)}}  =  \delta_{A_s} + \dsum_{T_i <  L(s)} \delta_{P_i} \,.
$$
Clearly, normalizing the infinite sequence $\bfV(s)$ by $s>0$ produces a random sequence which sums up to one.
An extraordinary observation of Pitman and Yor \cite{PiYo92} was that
\begin{equation} \label{PYIdent}
\frac{\bfV(s)}{s} \eind \frac{\bfV(S(t))}{S(t)}\,, \quad
\mbox{ for all } s,t \in (0,\infty)\,.
\end{equation} 
This is surprising, since the sequence on the right hand side is produced by ordering and scaling the points
$$
  P_i\,,\quad T_i \leq t,
$$
 and therefore has no special "last interval" as in \eqref{LastInt}.
Due to the  identity \eqref{PYIdent}, it suffices to describe the distribution of
$\bfV(1)$, thus we denote
$$
 \bfV(1) = \left( V_1(1),V_2(1),V_3(1),\ldots \right)=:  \left( D_1,D_2,D_3,\ldots \right)\,.
$$ 
The distribution of this sequence corresponds to the distribution
\COM{check} of the point process
$$
  M^{(\alpha)} = \dsum_{i=1}^\infty \delta_{D_i}\,.
$$
 It turns out to be easier to describe the distribution of the size--biased permutation of the  
 sequence $ \bfV(1) $, say
 $$
 \left( U_1,U_2,U_3,\ldots \right)\,,
$$
although clearly
$$
 \dsum_{i} \delta_{U_i}= \dsum_{i} \delta_{D_i}  =   M^{(\alpha)}  \,.
$$
Perman~\cite{PeThesis} proved that 
$$
 U_i = \xi_i \prod_{j=1}^{i-1} (1-\xi_j)\,,\quad j =1,2,3,\ldots
$$
for a sequence  of independent random variables $\xi_i=1,2,3,\ldots $, such that
$\xi\sim \mbox{Beta}(1-\alpha, i \alpha)$. We  call the distribution of the sequence $\bfV(1)$,  or equivalently of
the point process $M^{(\alpha)}$, the Pitman--Yor
distribution with parameter $\alpha$. This
distribution has further natural extension to two parameter family of distributions on the interval partitions, see Pitman and Yor~\cite{PiYo97}.
That family found important applications in nonparametric
Bayesian statistics, e.g. see Teh and Jordan~\cite{TeJo} and references therein. Moreover,
the arcsine laws for the fraction of time Brownian motion spends in the upper halfplane at a fixed time $t$ or at inverse local time $L(s)$, can be seen as corollaries of the results in~\cite{PiYo92}.

In the course of showing \eqref{PYIdent}, Pitman and Yor   showed 
that
$U_1$ above actually has the distribution of the final interval length $A_1$
from  \eqref{LastInt}. This distribution is the same as the generalized arcsine distribution 
of the random variable $A$ in theorem~\ref{thm_DL}.
These results allowed Perman~\cite{Pe93} to describe the density of 
$
 \sup \{ P_i \,: T_i < L(1)\} ,
$
which corresponds to the longest excursion of the $d$--dimensional Bessel process completed by the time 1, and
of
$
 D_1 = \max{\{  \sup \{ P_i \,: T_i < L(1)\}, A_1} \} ,
$
which has the same interpretation but  includes the last a.s. incomplete excursion.

\smallskip

 It is well known that   iid sequence  $(Y_n)$ satisfies \eqref{e:RegVar}, 
 if and only if  the following convergence of point processes holds
\begin{equation}\label{e:ConvPP}
 {N}_t:=\sum_{i\geq 1}\delta_{\left(\frac{i}{\widetilde{d}(t)},\frac{{Y}_{i}}{t}\right)}
 \cid 
N =\sum_{i\geq 1}\delta_{\left(T_i,P_i\right)}\,,
 \end{equation} 
 where $N$ denotes a Poisson process on the space $[0,\infty)\times (0,\infty]$ with intensity measure
 $\Leb \times \mu$, see Resnick~\cite{Re87}. The convergence of point processes 
 in \eqref{e:ConvPP} and throughout is to be understood  with respect to the vague topology
 on the space of Radon point measures on  $[0,\infty)\times (0,\infty)$, denoted by
 $M_p = M_p( [0,\infty) \times (0,\infty])$.

% 
%Using  \eqref{e:ConvPP}, for $\alpha \in(0,1)$ one can show that the following functional limit theorem holds
%\begin{equation}\label{e:ConvFLT}
%S_t(\;\cdot\;) = \sum_{i=1}^{\lfloor \widetilde{d} (t) \cdot\rfloor} \frac{Y_i}{t}
%\cid S (\;\cdot\;) = \sum_{T_i\leq  \cdot } Z_i \,,
%\end{equation}
%where $S$ is a strictly positive stable L\'evy process with  L\'evy  measure
%$\mu$, and the convergence takes place in $J_1$ topology on the space of
%c\`adl\`ag functions $D[0,\infty)$. 
%In particular 
%\begin{equation}\label{e:ConvCLT}
%\sum_{i=1}^{\lfloor \widetilde{d}(t) \rfloor} \frac{Y_i}{t}
%\cid S (1)\,,
%\end{equation}
%where $S(1)$ is strictly positive random variable with a stable law with the index $\alpha$, scale parameter $\sigma=1$, skewness parameter $\beta=1$ and  shift parameter $\mu=0$. 
%We 
%\COM{refer to Resnick~\cite{ResHTP}  for the proof?}

If \eqref{e:RegVar} and  \eqref{e:ConvPP} hold for a
general
stationary sequence $(Y_n)$, then it necessarily 
has the extremal index equal to 1, see Leadbetter et al. 
 \cite{LLR} for instance. In other words, the partial maxima in the sequence
 $(Y_n)$ behave as if the sequence was iid, i.e.  
 $ M_n = \max\{ Y_1, \ldots , Y_n \}$ satisfies $M_n/d_n \to \Phi_\alpha$, as $n\toi$ where $\Phi_\alpha$
 denotes the standard Fr\'echet distribution, i.e. $\Phi_\alpha(x) = \exp ( -x ^{-\alpha})$, $x>0$.
Next theorem, proved in the Appendix, claims that the opposite is also true.
Namely, strongly mixing sequences which satisfy \eqref{e:RegVar} and have extremal index equal to 1,
necessarily satisfy  \eqref{e:ConvPP}.
Observe that the theorem holds for all $\alpha>0$, and not merely on the interval
$(0,1)$ which is of our main interest in this paper.

\begin{theorem}\label{thm_PRM}
\COM{probably in Leadbetter and Rootzen!!!}
Suppose that $(Y_n)$ is  a stationary strongly mixing sequence of nonnegative regularly varying random variables with tail index $\alpha>0$.
\COM{with the tail index $\alpha \in (0,1)$?} 
Then
\begin{equation*}
N_t \cid N \,, 
\end{equation*}
as $t \toi$, where $N$ is $\mathrm{PRM}(\Leb\times\mu)$ if and only if 
 $(Y_n)$  has extremal index equal to 1.
%Moreover,  if we denote $W_\alpha=S_\alpha^{\leftarrow}(1)$, see  \eqref{e:SumProc}
%\begin{equation}\label{e:JointConv}
%\left({N}_t,\frac{\tau(t)}{\widetilde{d}(t)}\right)\cid ({N},W_\alpha)\,,
%\end{equation}
%as $t\toi$.
\end{theorem}

For iid steps, 
 \eqref{e:ConvPP} and the continuous mapping theorem 
imply
\begin{equation}\label{e:ConvFLT1}
S_t(\;\cdot\;) = \sum_{i=1}^{\lfloor \widetilde{d} (t) \cdot\rfloor} \frac{Y_i}{t}
\cid S (\;\cdot\;) = \sum_{T_i\leq  \cdot } P_i \,,
\end{equation}
 in $D[0,\infty)$  with respect to Skorohod's $J_1$ metric, see  Resnick~\cite{Re07}, Chapter 7, cf.
  theorem \ref{thm_Restriction} below.
Moreover,  for $\alpha\in(0,1)$ \COM{(see. Lemma 7.2.9. in Mikosch~\cite{MikoschNLIM}).} 
 \begin{equation*}
S(t) = \sum_{T_i\leq  t } P_i
  = \dint_{[0,t]\times (0,\infty]} y N(dt,dy) \,,
\end{equation*}
has finite value with probability 1 for all $t>0$.
Observe that %$S_t^{\leftarrow}(1) = {\tau(t)}/{\widetilde{d}(t)}$ or more generally
\begin{equation*}
{S_t^{\leftarrow}(1)  = \inf\left\{s:\sum_{i=1}^{\lfloor \widetilde{d}(t)s \rfloor}Y_i>t\right\}}=\frac{\inf\left\{k\in\mathbb{N}:\sum_{i=1}^{k}Y_i>t\right\}}{\widetilde{d}(t)}=\frac{\tau(t)}{\widetilde{d}(t)}\,.
\end{equation*}
Denote  $L^{(t)}=  S_t^{\leftarrow}(1) ={\tau(t)}/
{\widetilde{d}(t)}$   and recall
$L(u) =  S^{\leftarrow}(u) = \inf \{ t: S(t) > u \}$. For simplicity denote
$L =  L(1)$.
By an application of the
continuous mapping argument, from  \eqref{e:ConvFLT1} one can also show
\begin{equation*}\label{e:ConvRenewal}
L^{(t)} = \frac{\tau(t)}{\widetilde{d}(t)} = S_t^{\leftarrow}(1) 
\cid S^{\leftarrow}(1)= L \,.
\end{equation*}
In the following theorem we show that 
this convergence is joint with the convergence in 
\eqref{e:ConvFLT1},
whenever \eqref{e:ConvPP} holds. 
%In all such cases, this can be used to determine the asymptotic distribution of the steps $(Y_n)$  before the passage time ${\tau(t)}$, as we show in the next section.
%\COM{explain Mittag Lefler dist}
%
%\COM{pazi na indekse!!}

\begin{theorem}\label{thm_Restriction}
Suppose that $(Y_n)$ is  a stationary strongly mixing sequence of nonnegative regularly varying random variables
with extremal index equal to one and the tail index $\alpha \in (0,1)$.
Then, as $t\toi$	
\begin{equation}\label{JointConv}
 (N_t, S_t , L^{(t)}) \cid 
 (N, S , L)\,,
% \left(( S_t(\cdot)), L^{t}, {N}_t\Big|_{ \left[0,L^{t} \right]} \right)
% \cid
%  \left(( S(\cdot)), L , {N}\Big|_{\left[0,L \right]} \right)\,,
\end{equation}
in the product space $M_p \times D[0,\infty)  \times  \R $ and the corresponding  product topology (of  vague,  $J_1$ and Euclidean topologies).
\end{theorem}
\begin{proof}

By theorem~\ref{thm_PRM}
$$
 N_t \cid N\,,
$$
as $t\toi$.
We will first prove the convergence of the other two components
in \eqref{JointConv} by an application of the continuous mapping argument.
Since all the components are obtained by a transformation of the point process
 ${N}_t$, one can easily see that the convergence is joint. \COM{really?}

The proof of $S_t \cid S$ in $J_1$ topology is standard.
One could first observe that the functional $\psi^{\vep,t} :M_p \to D[0,t]$, given by
$$
 \psi^{\vep,t} (s)  = \sum_{t_i\leq s} x_i \1_{\{ x_i > \vep \}}\,,
 \quad \mbox{ for }\quad  m = \sum_i \delta_{t_i,x_i} \mbox{ and } s \in [0,t]\,,
$$
is a.s. continuous with respect to the distribution of the limiting point process $N$ and
chosen topologies.  Then one can simply follow the lines of the proof of Theorem 7.1 in Resnick~\cite{Re07}, and finally
apply lemma 16.3 in Billingsley~\cite{Bi99} to extend the convergence from
$D[0,t]$ to $D[0,\infty]$.

By the
continuous mapping argument, see Lemma~\ref{Lema2} in the Appendix, it follows that
\begin{equation*} \label{e:ConvRenewal2}
L^{(t)} = S_t^{\leftarrow}(1) 
\cid S^{\leftarrow}(1)= L \,.
\end{equation*}
%\qed

\end{proof}

 The random variable $L$ in \eqref{JointConv} represents the first passage time of the level one by the $\alpha$--stable subordinator $S$. Its distribution is known
 in the literature as 
a Mittag--Leffler  distribution.

\section{Main theorem} \label{Sec:main}

Our main result extends the sufficiency part  of the  Dynkin--Lamperti theorem in a couple of ways.
We first show that one can describe the limiting distribution of not merely the age of
the renewal process at time $t$, but also the behavior of all other large steps  before that time. By doing that, we  obtain the Pitman--Yor distribution as
the limiting distribution for the steps after appropriate normalization.
We also show 
 that the statement of   Dynkin--Lamperti theorem about iid regularly varying random variables
  can be generalized to cover all
regularly varying sequences with non--clustering extremes considered in the previous section. 
For simplicity, denote
$$
 A^{(t)}  = 1- S_t\left(L^{(t)} - \right) = 
 %=  \left( 1 -  \sum_{i < \tau (t) } \frac{Y_i}{t}\right) =
  \frac{ t -  \sum_{i < \tau (t) } {Y_i}}{t}
$$
and  $A= A_1=  1- S\left(L - \right)$. 
%Moreover, denote
% $$
% M^{(\alpha,t)} = \delta_{A^{(t)} } +  \sum_{i < \tau (t) } \delta_{ {Y_i}/{t}} 
%$$

\begin{theorem}\label{thm_DLextended}
Suppose that $(Y_n)$ is  a stationary strongly mixing sequence of nonnegative regularly varying random variables
with extremal index equal to one and the tail index $\alpha \in (0,1)$. Then,
as $t\toi$, 
\begin{equation}\label{e:MT}
 \delta_{A^{(t)} } +  \sum_{i < \tau (t) } \delta_{ {Y_i}/{t}} 
 %M^{(\alpha,t)}  
 \cid
 M^{(\alpha)}\,,
% \sum_{T_i < L } \delta_{ P_i} + \delta_{A}\,, 
\end{equation}
where $ M^{(\alpha)}$ represents a Pitman-Yor point process with parameter $\alpha$.
Moreover, the convergence above is joint with
$$
  A^{(t)} \cid  A\,,
$$
as $t\toi$,
%$$
%1- S_t\left(L^{(t)} - \right) =  \left( 1 -  \sum_{i < \tau (t) } \frac{Y_i}{t}\right)
% \cid
% 1- S\left(L - \right)=A_1 \,,
%$$
% $$
%  1- S_t\left (\frac{\tau_t }{\widetilde{d}(t)} - \right) =  \left( 1 -  \sum_{i < \tau (t) } \frac{Y_i}{t}\right)  \cid
%  1- S(W_\alpha -)\,,
%$$
%\COM{indeksi $\alpha, t$ nisu isti izvadi $ S_t \left (\frac{\tau_t } {\widetilde{d}(t)} -\right)$ ispred}
where $A$ has the generalized arcsine distribution given in \eqref{ArcSDens}.
\end{theorem}
%\COM{sakrti ruzne dijelove, definiraj PY point process}

For a measure $\nu$ on a measurable space $(\Sb,\Sc)$, by $\nu\vert_{B}$ we denote 
the restriction of the measure $\nu$ on the set $B\in \Sc$ given by $\nu\vert_{B}(C) =  \nu (B\cap C),$
 $C \in \Sc$. Abusing this  notation somewhat, for any time period $A\subseteq [0,\infty)$
and an arbitrary  point measure $n \in M_p([0,\infty)\times(0,\infty])$, we write
\begin{equation}\label{e:NtRestr}
n\Big|_{A} 
\ \mbox{for }\ 
%N_t\Big|_{A\times(0,\infty]}\, \ 
n\Big|_{A\times(0,\infty]}\,.
%N_t\Big|_{\left[0,\frac{\tau(t)}{\widetilde{d}(t)}\right]}: =
%N_t\Big|_{\left[0,\frac{\tau(t)}{\widetilde{d}(t)}\right]\times(0,\infty]}\,,
\end{equation}
 
\begin{proof}
By theorem~\ref{thm_Restriction}, $(N_t,S_t) \cid (N,S)$ in the appropriate
product topology, as $t\toi$.
Moreover, the limit $(N,S)$ a.s. satisfies the regularity assumption
of lemma~\ref{Lema2} and theorem~\ref{thm_contin} below. Therefore, this convergence is
joint with the convergence in $$A^{(t)} \cid A\,.$$ 
By  theorem~\ref{thm_contin}, as $t\toi$,
$$
 N_t\big|_{[0,L^{(t)})} = \dsum_{i < \tau(t)} \delta_{\left(\frac{i}{\widetilde{d}(t)},\frac{{Y}_{i}}{t}\right)}
 \cid
N \big|_{[0,L)} =\sum_{T_i <  L}\delta_{\left(T_i,P_i\right)}\,.
$$
In particular for
$f \in C_{K}^{+}((0,\infty])$, where $C_{K}^{+}$ denotes the family of nonnegative continuous functions with compact support
$$
E \left[ \exp \left\{ - \sum_{i < \tau(t)} f (Y_i /t) \right\} \right]
\to 
E \left[ \exp \left\{ - \sum_{T_i < L} f (P_i) \right\} \right]\,,
$$
 as $t\toi$.
Since the corresponding Laplace functionals converge, we conclude that
$$
 \sum_{i < \tau (t) } \delta_{ {Y_i}/{t}} 
 \cid \sum_{T_i < L} \delta_{P_i}\,.
$$
Because, this holds jointly with $A^{(t)} \cid A$. We conclude that
$$
\delta_{A^{(t)} } + \sum_{i < \tau (t) } \delta_{ {Y_i}/{t}} \cid
\delta_A +  \sum_{T_i < L} \delta_{P_i} = M^{(\alpha)}\,.
% \sum_{T_i < L} \delta_{P_i}+
% \delta_{1- S_t \left (\frac{\tau_t } {\widetilde{d}(t)} -\right)  }\,.
$$
%\qed

\end{proof}
 \COM{check this all}

\begin{remark}
 The strong mixing assumption in the theorem is actually unnecessarily strong,
 one could alternatively consider any stationary sequence $(Y_n)$ which satisfies \eqref{e:ConvPP}.
 In the extreme value theory  it is known that this holds under milder conditions
  cf. Basrak et al.~\cite{BKS}.
%  Instead of the assumptions above,
%  it is enough to suppose that there exists a sequence $(r_n)$, $r_n\toi$, $r_n = o(n)$ 
%  such that \eqref{e:mixcon} and \eqref{e:anticluster} hold, 
%  as clear from the proof of theorem~\ref{thm_PRM}, cf. Basrak et al.~\cite{BKS}.
%  Those assumptions (or versions thereof) appeared in the literature repeatedly, again see \cite{BKS}
%  and references therein.
\end{remark}
\begin{remark}  
 An interesting implication of  theorem~\ref{thm_DLextended} concerns 
  the lengths of excursions  of the simple symmetric random walk during the first $n$ steps. They are known to be independent
  and regularly varying with index $\alpha=1/2$. Therefore, theorem can be applied to 
  deduce and extend results in Cs\'aki and Hu~\cite{CsHu} about the asymptotic distribution of
these excursions. %, cf. Basrak and \v Spoljari\'c~\cite{BaSp}.
\end{remark}

The value $  A^{(t)} $ in theorem~\ref{thm_DLextended} is called the undershoot or  the age of the renewal process at time $t$. Similarly, one could define the overshoot at $t$ as
$$
 B^{(t)}  = S_t\left(L^{(t)}\right) - 1 = 
  \frac{   \sum_{i \leq  \tau (t) } {Y_i} - t}{t}\,.
$$
Recall that $L^{(t)}$ represents the scaled first passage time. Straightforward application of
theorem~~\ref{thm_Restriction} and  lemma~\ref{Lema2} yields the following corollary
which should be compared with Dynkin--Lamperti theorem, cf.
theorem  8.6.3 of  Bingham et al.~\cite{BGT}.
Note however that the corollary admits weak dependence between the steps 
of the renewal process.

\begin{corollary}
Under the assumptions of theorem~\ref{thm_DLextended}, the convergence in \eqref{e:MT}
is joint with 
$$
 \left(A^{(t)},B^{(t)},L^{(t)} \right)
 \cid \left(A,B,L \right)\,,
$$
as $t\toi$. Moreover, the joint density of the random vector $(A,B)$ is
given in Bingham et al.~\cite{BGT} theorem 8.6.3, while the random variable $L$ has
the same distribution as
in theorem~~\ref{thm_Restriction}.

\end{corollary}

\section{Appendix}

\begin{proof} (of theorem~\ref{thm_PRM})
%\COM{check Danijel but give credit to BKS or thesis only, make into a LEMMA}
As we explained before the theorem, it remains to show sufficiency. Assume that $(Y_n)$ is strongly mixing with the extremal index equal to 1.

Denote by $\alpha(n)$ the mixing coefficients of the sequence $(Y_n)$.
Then set $l_n = \lfloor \max\{1, n^{0.1}\} \rfloor $,
clearly $l_n = o(n)$, $l_n\toi$. Introduce also the sequence
$r_n = \lfloor \max\{1, n \sqrt{\alpha (l_n)}, n^{2/3} \} \rfloor$,
and observe $r_n \toi $. By proposition 1.34 in Krizmani\'c \cite{KrThesis},
the sequence $(r_n)$ satisfies the following condition:
%($\mathcal{A}'(a_n)$)
%\label{c:mixcond}
for every $f \in C_{K}^{+}([0,\infty) \times
(0,\infty))$ 
\COM{check this repeatedly, Danijel's thesis too, recall $C^+_K$}
\begin{equation}\label{e:mixcon}
 E \biggl[ \exp \biggl\{ - \sum_{i=1}^{\infty} f \biggl(\frac{i}{n}, d_n^{-1}{Y_{i}}
 \biggr) \biggr\} \biggr]
 - \prod_{k=1}^{\lfloor L n /r_n \rfloor} \Exp \biggl[ \exp \biggl\{ - \sum_{i=1}^{r_{n}} f \biggl(\frac{kr_{n}}{n}, d_n^{-1}{Y_{i}} \biggr) \biggr\} \biggr] \to 0.
\end{equation}
as $n \to
\infty$, assuming  without loss of generality that the support of $f$  lies in $[0,L]\times [l,\infty]$ for  $L, l >0$. In other words, the strong mixing condition implies the condition $\mathcal{A}'(a_n)$
introduced in Basrak et al.~\cite{BKS}.

%\COM{The part (i)  of the non--clustering condition holds thus}

Observe that 
\begin{equation} \label{e:TailPr}
\lim_{n\toi} (P(Y \leq d_nu))^n = \lim_{n\toi} \left(1 - \frac{n P(Y > d_nu)}{n}\right)^n  = 
e^{-u^{-\alpha}}>0\,, 
\end{equation}
for any $u>0$.
Note that the sequences $(l_n)$ and $(r_n)$  satisfy  $r_n = o(n)$
$n\alpha(l_n) = o(r_n)$, $l_n = o(r_n)$. According to O'Brien \cite{OB87},
the extremal index $\theta$  of the sequence $(Y_n)$ satisfies
$$
  \theta = \lim_{n\toi} P(M_{r_n} \leq d_n u\mid Y_0 > d_n u) 
$$
for any fixed $u>0$. Since, by assumption, $\theta = 1$, we obtain
$$
 P(M_{r_n} > d_n u \mid Y_0 > d_n u ) = P \biggl( \max_{1 \le i \le r_{n}} {Y_{i}} > d_n u\,\bigg|\,{Y_{0}}>d_n u \biggr) 
  \to 0\,,
$$
as $n\toi$.
Hence, by stationarity, for every $u > 0$, \COM{check}
\begin{equation}
\label{e:anticluster}
 \limsup_{n \to \infty}
  P \biggl( \max_{1 \le |i| \le r_{n}} {Y_{i}} > d_n u\,\bigg|\,{Y_{0}}>d_n u \biggr) = 0.
\end{equation}
Consequently $(Y_n)$ is jointly regularly varying in the sense
of Basrak and Segers~\cite{BaSe}. Moreover,   its tail sequence is trivial.
% Roughly speaking, the extremes in this sequence come in isolation. 
%there is no clustering of extremes.
% in that sequence and the extremal index of
%the sequence $(Y_i)$ equals 1, see Resnick ?? or \cite{BKS}.

We observe next that by \eqref{e:mixcon} and \eqref{e:anticluster},  the point processes
\begin{equation}\label{e:tildePP}
 \widetilde{N}_n:=\sum_{i\geq 1}\delta_{\left(\frac{i}{n},\frac{{Y}_{i}}{d_n}\right)}
 \end{equation} 
 in $M_p$ satisfy the assumptions of Theorem 2.3 in  \cite{BKS}, adjusting
the state space from $[0,1] \times (0,\infty]$ used there, to the case  $[0,L] \times (0,\infty]$
we need here.
This means, in particular,
\COM{Leb measure denote uniformly}
that  for all $L>0$, $u>0$
\[
   \widetilde N_n  \bigg|_{[0, L] \times (u,\infty]} 
   \cid N^{(u)}
    = \sum_i \delta_{(T^{(u)}_i, u Z_{i})} \bigg|_{[0, L] \times (u,\infty]}\,,
\]
as $n \to \infty$, where 
 $$\sum_i \delta_{T^{(u)}_i, Z_i}$$
  is a homogeneous Poisson process on $[0,\infty) \times (1,\infty] $ with intensity 
 $u^{-\alpha}  \Leb \times \widetilde \mu $,
 where  $\widetilde \mu $ denotes the probability measure obtained by restricting measure $\mu$ to the interval $(1,\infty]$. However, if we denote  by $N = \sum_i \delta_{T_i, Z_i}$ a Poisson process on $[0,\infty) \times (0,\infty] $ with intensity 
 $ \Leb \times  \mu $, then for 
   for all $L>0$, $u>0$
$$
N^{(u)}
    \eind  N %=  \sum_i \delta_{(T_i, Z_{i})} 
    \bigg|_{[0, L] \times (u,\infty]}\,.
$$
In particular,
\[
   \widetilde N_n  
   \cid N \,.
\]
 
Consider the mapping from $M_p \times (0,\infty)^2$ to $M_p$, given by
  \begin{equation}\label{e:muv1}
   (m,a,b) \mapsto m^{a,b},
  \end{equation}
  where
$$
m^{a,b} (I \times J) = m (I/a \times J/b)\,,
$$  
for all measurable sets $I,J$.
Vague convergence theory as presented in section 3.4 of Resnick~\cite{Re87},  shows that this mapping is continuous. 
This turns to be useful, since
\begin{equation} \label{e:veza}
N_t (I, J ) = \widetilde{N}_{\lfloor \widetilde{d}(t) \rfloor}
 \left( \frac{\lfloor \widetilde{d}(t) \rfloor}{ \widetilde{d}(t) }  I \times \frac{d( \widetilde{d}(t))}{t} J
 \right)\,.
\end{equation}
 Observe that 
 $$
 \widetilde{N}_{\lfloor \widetilde{d}(t) \rfloor} \cid N \ \mbox{ and } \ 
 \left(  \frac{\lfloor \widetilde{d}(t) \rfloor}{ \widetilde{d}(t) },
 \frac{d( \widetilde{d}(t))}{t} \right) \to (1,1)
 $$
 as $t \toi$.  
% with respect to the distribution of the point process $N$ 
 Hence, by \eqref{e:veza} and \eqref{d_asympt_eq}, one can conclude
 \COM { when applied on measures restricted on the
 sets of the form $[0,L] \times (u,\infty]$,  $L,u>0$ }
$$
 N_t \cid N\,,
$$
as $t \toi$ as well.
%\qed
  
\end{proof}

Denote by $n, n_t, \  t>0$, arbitrary Radon
point measures in $M_p([0,\infty)\times (0,\infty]$.
One  can always write
$$
 n_t = \sum_i \delta_{v^t_i,y^t_i},\  n = \sum_i \delta_{v_i,y_i}
$$
for some sequences $(v^t_i)$, $(y^t_i)$, $(v_i)$ and $(y_i)$ of  positive real numbers.
Denote further the corresponding cumulative sum functions of the point measures $n, n_t, \  t>0$, by
$$
 s_t(u) = \dint_{[0,u]\times(0,\infty)} y\, n_t(dv,dy)\,,\ u \geq 0 \,,
$$
and
$$
 s(u) = \dint_{[0,u]\times(0,\infty)} y\, n(dv,dy)\,,\ u \geq 0 \,,
$$
Assume that these values are finite for each $u>0$, but tend to $\infty$ as $u\toi$. This makes 
$s_t$ and $s$ well defined, unbounded,
nondecreasing elements of the space of c\`adl\`ag functions $D[0,\infty)$.
%endowed with Skorohod's $J_1$ topology. 
Their right continuous generalized inverses (or hitting time functions) we denote
by $s^{\leftarrow}$ and $s_t^{\leftarrow}$, recall that
%\begin{equation*}
$s^{\leftarrow}(u)=\inf\{v\in[0,\infty):s(v)>u\} \,,\ u \geq 0\,.$
%\end{equation*}
We will use the following abbreviations in the sequel
\begin{equation*}
 \tau_t =  s_t^{\leftarrow}(1) \mbox{ and } \tau=s^{\leftarrow}(1)\,.
\end{equation*}
It is well known that $J_1$ convergence 
$$
 s_t \cij s\,, % \mbox{ together with } \tau_t \to \tau 
$$
in general does not imply convergence of $s_t$ towards $s$ at a given point. However,
the following technical lemma shows that under some regularity conditions, it 
implies the convergence of both $s_t$ and its left limit at the first passage time $\tau_t$. 
It is a consequence of Theorem 13.6.4  in Whitt~\cite{Wh02}  
\COM{check numbering in WHitt}
which has a weaker assumption that  $s_t$ converge towards $s$ in $M_2$ topology.
\begin{lemma} \label{Lema2} %(continuity of age and residual life functionals)
Assume that  \COM{1st coord. not needed?}
\begin{equation} \label{e:convof4}
s_t \cij s \,,
\end{equation}
and suppose that  $s(v)<1 $ for each $v < \tau$. Then
\begin{equation}
 (\tau_t,s_t(\tau_t-),s_t(\tau_t) ) \to  (\tau, s(\tau-) ,  s(\tau)) \,.
\end{equation}
\end{lemma}

\begin{theorem} \label{thm_contin} %(continuity of age and residual life functionals)
Assume that  
\begin{equation} \label{e:convof2}
 (n_t, s_t)
  \to (n, s)\,,
\end{equation}
in the product topology (of vague and $J_1$  topologies) as $t\toi$.
Assume further that 
$
 0<s(\tau-)<1 <s(\tau)  \mbox{ and }  n (\{v\}\times (0,\infty))  \leq 1
$
for all $v \geq 0$.
Then 
 \begin{equation}\label{NRestric}
  n_t\big|_{[0,\tau_t]}  \civ 
  n\big|_{[0,\tau]} \,
\quad
\mbox{ and }  \quad
  n_t\big|_{[0,\tau_t)}  \civ 
  n\big|_{[0,\tau)} \,.
\end{equation}
\end{theorem}

\begin{proof}

% Take $f \in C_{K}^{+}([0,\infty) \times
%(0,\infty))$, it necessarily has a support contained in $[0,\infty) \times (u,\infty]$ for some fixed $u>0$.

By the definition of the vague  convergence, it is enough to show the
convergence in the state space $[0,\infty) \times (u,\infty]$
for some number $u>0$ in the arbitrary neighborhood  of 0.
 Because $n$ is a Radon measure,  one can always find
  $u>0$ which is arbitrarily close to 0 and satisfies
$n([0,\infty) \times \{u\}) =0$ and
$u < s(\tau) - s(\tau-)$. Since $s$ is a c\`adl\`ag function, there exists $\vep_0>0$, such that
\begin{equation} \label{pomoc1}
s(\tau+\vep_0) - s(\tau) < u/3\,.
\end{equation}

Since $s_t$ and $s$ are monotone functions, by theorem 2.15 a) in Jacod and Shiryayev~\cite{JaSh}, chapter VI,  
\eqref{e:convof2} implies that there exists a dense set of points $G\subseteq [0,\infty)$
such that  $s_t(d) \to s(d)$ for each $d \in G$. 
Take $0< \vep \leq \vep_0$,  such that $\tau + \vep \in G$,
$n(\{\tau+ \vep \} \times (0,\infty) ) =0$, and such that
\COM{ uvjet $n(\{\tau+ \vep \} \times (0,\infty) ) =0$ treba}
\begin{equation} \label{pomoc2}
n\big|_{[0,\tau]\times (u,\infty)} = 
n\big|_{[0,\tau+\vep]\times (u,\infty)} =
\sum_{i=1}^k \delta_{v_j, y_j}\,.
\end{equation}
By Proposition 3.13 in Resnick~\cite{Re87}, there exists a constant $t'_0$, such that for all $t>t'_0$
\begin{equation*}
n_t\big|_{[0,\tau+\vep ]\times (u,\infty)} = 
\sum_{i=1}^k \delta_{v^t_j, y^t_j}\,.
\end{equation*}
and 
\begin{equation} \label{ConOfPoints}
(v^t_j, y^t_j) \to (v_j, y_j)\,,\quad j =1,\ldots ,k \,,
\end{equation}
as $t\toi$. 
Since $\tau_t \to \tau$ by lemma~\ref{Lema2}, there also exist $t''_0$, such that for all $t>t'_0$
\begin{equation*}
\tau_t < \tau +\vep\,. 
\end{equation*}
For $t > t'_0 \vee t''_0$
\begin{equation}
n_t\big|_{[0,\tau+\vep ]\times (u,\infty)} = 
n_t\big|_{[0,\tau_t]\times (u,\infty)}  + 
n_t\big|_{(\tau_t,\tau+\vep]\times (u,\infty)} \,.
\end{equation}
But we will show that the second point process on the right hand side above
equals zero for all $t$ large enough.

Because $s_t(\tau+\vep) \to  s(\tau+\vep)\,,\ s_t(\tau_t) \to  s(\tau)$,  by \eqref{pomoc1}
there exists $t'''_0$ such that  for $t> t'''_0$
\begin{equation}
 s_t( \tau+\vep ) - s_t (\tau_t)  < \frac 23 u.
\end{equation}
\COM{ not $\leq  2 | s( \tau+\vep ) - s (\tau)| $}

Now for $t >  t'_0 \vee t''_0 \vee t'''_0$
\begin{eqnarray}
\lefteqn{ n_t \left({(\tau_t,\tau+\vep]\times (u,\infty)} \right) 
= \dsum_{v^t_i \in (\tau_t,\tau+\vep]} \1_{\{y^t_i > u\}}}\\
 & \leq &
\frac 1u  \dsum_{v^t_i \in (\tau_t,\tau+\vep]} y^t_i  \1_{\{y^t_i > u\}}
\leq \frac{1}{u} \left( s_t( \tau+\vep ) - s_t (\tau_t)  \right) < \frac 23\,.
\end{eqnarray}
Since $n_t$ is a point measure, $n_t \left({(\tau_t,\tau+\vep]\times (u,\infty)} \right)  =0$ for all $t$ large enough.
So, by  Proposition 3.13 in \cite{Re87} 
\begin{equation} \label{conv:pointmeas}
  n_t\big|_{[0,\tau_t]}  \civ
  n\big|_{[0,\tau]} \,.
\end{equation}

Suppose that the rightmost point of the point measure on the right hand side above
has the coordinates
$$
 (\tau, j) \mbox{ where } j = s(\tau)-s(\tau-) > 0\,,
$$
because $ n (\{\tau \}\times (0,\infty)) \leq 1 $
Then, by  \eqref{ConOfPoints}  there exist
points $(\sigma_t,j_t)$ of the point measures on the left hand side of \eqref{conv:pointmeas}
such that
$$
(\sigma_t,j_t) \to 
(\tau,j) = (\tau, s(\tau)-s(\tau-) )\,.
$$

By the assumption, there exists $\vep >0$ such that $s(\tau) > 1+\vep$.
 Note that there also exists $d \in G$, $d< \tau$ such that
 $$
  s(\tau - )  - \frac{\vep}{4} < s(d)\,. 
 $$
 For all sufficiently large $t$ now $\sigma_t>  d$,  and
 \begin{equation*}
{ s_t(\sigma_t) \geq s_t (d) + j_t > s(d) - \frac{\vep}{4} + j - \frac{\vep}{4}   }
>   s(\tau-)  + j  - \frac{3}{4}\vep >1\,.
 \end{equation*}
Therefore  $\sigma_t = \tau_t$ for all sufficiently large $t$.
Recall that we fixed a constant $u>0$ such that $n([0,\infty) \times \{u\}) =0$.
Then for all large enough $t$
\begin{equation*}
  n_t\big|_{[0,\tau_t)\times (u,\infty]}  = 
  n_t\big|_{[0,\tau_t]\times (u,\infty]}  - \delta_{\tau_t,j_t}  \,.
\end{equation*}

Thus \eqref{conv:pointmeas} together with $(\tau_t,j_t) \to (\tau,j)$,   implies
\begin{equation*}
  n_t\big|_{[0,\tau_t)\times (u,\infty]}  = 
  n_t\big|_{[0,\tau_t]\times (u,\infty]}  - \delta_{\tau_t,j_t}  \civ 
  n\big|_{[0,\tau]\times (u,\infty]}- \delta_{\tau,j}  =  n\big|_{[0,\tau)\times (u,\infty]} \,,
\end{equation*}
as $t\toi$.
%\qed

\end{proof}
\COM{don't worry about possibility of two points on the line, there is some additional text in earlier versions}

\section*{Acknowledgements}
This research was supported in part by Croatian Science Foundation under the project 3526.

\bibliographystyle{abbrv}
\bibliography{bbbib_initials}{}

\end{document}